\documentclass[12pt, twoside]{article}
\usepackage{amsmath,amsthm,amssymb}
\usepackage{times}
\usepackage{enumerate}
\usepackage{bm}
\usepackage[all]{xy}
\usepackage{mathrsfs}
\usepackage{amscd}
\usepackage{color}
\def\titlerunning#1{\gdef\titrun{#1}}
\makeatletter
\def\author#1{\gdef\autrun{\def\and{\unskip, }#1}\gdef\@author{#1}}
\def\address#1{{\def\and{\\\hspace*{18pt}}\renewcommand{\thefootnote}{}%
		\footnote {#1}}%
	\markboth{\autrun}{\titrun}}
\makeatother
\def\email#1{e-mail: #1}
\def\subjclass#1{{\renewcommand{\thefootnote}{}%
		\footnote{\emph{Mathematics Subject Classification (2020):} #1}}}
\def\keywords#1{\par\medskip
	\noindent\textbf{Keywords.} #1}

\newtheorem{theorem}{Theorem}[section]
\newtheorem{corollary}[theorem]{Corollary}
\newtheorem{lemma}[theorem]{Lemma}
\newtheorem{proposition}[theorem]{Proposition}
\theoremstyle{definition}

\newtheorem{remark}[theorem]{Remark}

\numberwithin{equation}{section}

\xyoption{all}

\def \C {\mathbb{C}}

\def \a {\alpha }
\def \b {\beta}

\def \De {\Delta}
\def \la {\lambda}
\def \La {\Lambda}

\def\Om{\Omega}

\def\na {\nabla}

\begin{document}
	\baselineskip=17pt
	
	\titlerunning{Curvature operator and Euler number}
	\title{Curvature operator and Euler number}
	
	\author{Teng Huang and Qiang Tan}
	
	\date{}
	
	\maketitle
	
	\address{T. Huang: School of Mathematical Sciences, University of Science and Technology of China; CAS Key Laboratory of Wu Wen-Tsun Mathematics, University of Science and Technology of China, Hefei, Anhui, 230026, P. R. China; \email{htmath@ustc.edu.cn;htustc@gmail.com}}
	\address{Q. Tan: School of Mathematical Sciences, Jiangsu University, Zhenjiang, Jiangsu 212013, People’s Republic of China; \email{tanqiang@ujs.edu.cn}}

	\begin{abstract}
	Let $(X,g)$ be a compact $n$-dimensional smooth Riemannian manifold with a lower bound on the average of the lowest $n-p$ eigenvalues of the curvature operator and the diameter of $X$ is bounded above by $D>0$. In this article, we investigate the relationship between the curvature operator and the Euler number of $X$. Our analysis is based on more general vanishing theorems for a Dirac operator associated with a smooth $1$-form on $X$. As a consequence, we obtain partial affirmative answers to Question 4.6 posed by Herrmann, Sebastian, and Tuschmann in \cite{HST}. Specifically, we prove that if a compact $2m$-dimensional manifold admits an almost nonnegative curvature operator (ANCO) and has a nontrivial first de Rham cohomology group, then its Euler number vanishes. Furthermore, in the case where $m=2$, we show that the Euler number is nonnegative. This result provides a complete resolution to their question in the four-dimensional setting.
\end{abstract}

\section{Introduction}
Let $(X,g)$ be a compact $n$-dimensional smooth Riemannian manifold with $n\geq3$. Understanding the relationship between curvature and the topology of a Riemannian manifold is a fundamental theme in Riemannian geometry. Specifically, investigating the connections between Ricci curvature and the topology of a Riemannian manifold is one of the central topics in differential geometry. A classical theorem by Bochner states that if the Ricci curvature $Ric(g)$ of $X$ is nonnegative, then the first Betti number $b_{1}(X)$ of $X$ satisfies $b_{1}(X)\leq n$. Moreover, equality holds if and only if $X$ is isometric to a flat torus.
	
Let $\la_{1}\leq\cdots\leq \la_{\binom{n}{2}}$ denote the eigenvalues of the curvature operator of $(X,g)$. Berger \cite{Ber61} and Meyer \cite{Mey} established vanishing results for the Betti numbers of manifolds with positive curvature operators, i.e., $\la_{1}>0$, and in particular Meyer showed that they are rational (co)homology spheres. A smooth compact  manifold $X$ is said to admit almost nonnegative curvature operator (ANCO) if there exists a sequence of Riemannian metrics $\{g_{i}\}$ on $X$ such that all eigenvalues
	$\la_{i}=\la(g_{i})$ of the associated curvature operator and the diameter ${\rm{diam}}(g_{i})$ of $g_{i}$ satisfy
	$$\la_{i}\geq -\frac{1}{i},\ {\rm{diam}}(g_{i})\leq 1.$$
	This definition was introduced by Herrmann-Sebastian-Tuschmann in \cite{HST}.	  In \cite{HST}, the authors proposed the following question:\\
	\textbf{Question 4.6} Do manifolds with almost nonnegative curvature operator have nonnegative Euler number?
	
	Fix constants $\kappa< 0$, $D>0$, $\La>0$ and $1\leq p\leq\lfloor\frac{n}{2}\rfloor$.  Suppose that
	$$\frac{\la_{1}+\cdots+\la_{n-p}}{n-p}\geq\kappa,\ {\rm{diam}}(X)\leq D,\ and\ -\kappa D^{2}=\La^{2}.$$
	Peter and Wink \cite{PW} established a more general vanishing and estimation theorem for the $p$-th Betti number of compact $n$-dimensional Riemannian manifolds with $\la_{1}+\cdots+\la_{n-p}\geq0$. As a consequence, they proved that manifolds with  $\lceil\frac{n}{2}\rceil$-positive curvature operators are rational homology spheres.  In \cite{Yu}, the author studied the upper bound of the Betti numbers of a compact Riemannian manifold given integral bounds on the average of the lowest eigenvalues of the curvature operator. 
	
	Throughout our article, we denote by $\mathfrak{X}(p,\kappa,D,\La)$  the class of compact Riemannian $n$-manifolds $X$, where
	\begin{itemize}
		\item the lower bound on the average of the lowest $n-p$ eigenvalues of the curvature operator is greater than or equal to $\kappa$;
		\item  the diameter satisfies $D(g)\leq D$; 
		\item  $-\kappa D^{2}=\La^{2}$.
	\end{itemize}
Every compact smooth Riemannian manifold belong to $\mathfrak{X}(p,\kappa,D,\La)$ for some $\kappa$ and $D$.

Based on work of Li \cite{Li}, Petersen and Wink \cite{PW} generalized the Bochner technique and obtained estimation results for the Betti numbers of the compact Riemannian manifold $X\in\mathfrak{X}(p,\kappa,D,\La)$. The case of $\la_{1}\geq\kappa$ was also proved by Gallot \cite{Gal}. In \cite{PW}, the authors proved that for sufficiently small $\La$, the Betti numbers of  $X\in\mathfrak{X}(p,\kappa,D,\La)$ satisfy $b_{k}(X)\leq\binom{n}{k}$ for all $k\leq p$ . Gromov \cite{Gro81} established similar bounds on the Betti numbers in the context of sectional curvature.

However, these estimates of Betti numbers do not provide information about the Euler number $\chi(X)$ of $X$. This leads us to a natural question:\\
\textbf{Question} 
	what is the Euler number of the $2n$-dimensional compact Riemannian manifold $X\in\mathfrak{X}(p,\kappa,D,\La)$ when $\La$ is sufficiently small? 
	In particular, does the manifold $X\in\mathfrak{X}(1,\kappa,D,\La)$ with $\La$ small enough necessarily  have $\chi(X)\geq0$?
	
This question generalizes the Question 4.6 proposed by Herrmann-Sebastian-Tuschmann. Our first mail result in this article establishes a connection between the eigenvalues of the curvature operator and the Euler number of $X$.
\begin{theorem}\label{C7}
Let $(X,g)$ be a compact $2n$-dimensional smooth Riemannian manifold with nonzero first de Rham cohomology group. There exists a uniform positive constant $C(n)$ such that if $X\in\mathfrak{X}(p,\kappa,D,\La)$ satisfies 
$$\La\leq C(n),$$
then the Euler number of $X$ satisfies\\
(1) $\chi(X)=0$, when $p=n$;\\
(2) $(-1)^{n}\chi(X)\geq0$, when $p=n-1$.
\end{theorem}
\begin{remark}
The assumption of a nontrivial first cohomology group in Theorem \ref{C7} is essential. For example, let $(X,g)$ be a compact $2n$-dimensional Riemannian manifold with $n$-positive curvature operator. According to \cite[Theorem A]{PW}, the Betti number $b_{p}(X)=0$ for $0<p<2n$, so the Euler number of $X$ satisfies $\chi(X)=2$.  However, such $X$ is in $\mathfrak{X}(n,0,D,0)$.
\end{remark}
While our results provide only partial affirmative answers to Question 4.6 proposed by Herrmann-Sebastian-Tuschmann in general dimensions, we obtain a complete resolution in the $4$-dimensional case.
\begin{corollary}\label{C1}
	Let $(X,g)$ be a compact $4$-dimensional smooth Riemannian manifold. There exists a uniform positive constant $C$ such that if $X\in\mathfrak{X}(1,\kappa,D,\La)$ satisfies
	$$\La\leq C,$$
	then the Euler number of $X$ satisfies $\chi(X)\geq0$. 
\end{corollary}
\begin{remark}
Actually, we observe that when dealing with 1-forms, we only need to use the properties of Ricci curvature. In fact, if the Ricci curvature of a compact manifold $X$ satisfies $$Ric(g)D^{2}(g)\geq -C,$$
where $C$ is a uniform positive constant, then there exists some $t\in\mathbb{R}$, $t\neq 0$ such that $H^{1}(X,t\theta)=0$. By a similar argument as in Corollary \ref{C1}, we can also prove that $\chi(X)\geq0$ when $\dim X=4$.
\end{remark}
For a smooth Riemannian manifold $(X,g)$, let ${\rm{Iso}}(X,g)$ denotes the group of Riemannian isometries $f:(X,g)\rightarrow(X,g)$. The Lie algebra of ${\rm{Iso}}(X,g)$ is spanned by the Killing vectors of $(X,g)$. A classical theorem of Bochner asserts that the isometry group of a compact Riemannian manifold with negative Ricci curvature is finite (see \cite{Boc}).  Several authors have also attempted to estimate the order of ${\rm{Iso}}(X,g)$  under the assumption of negative Ricci curvature (see \cite{DSW,Kat,KK}). 

We denote by $\theta^{\sharp}$ the Killing vector field on $X$, and the dual $1$-form $\theta$ with respect to $\theta^{\sharp}$ is always non-closed. In the second part of our article,  we will focus on the study of Killing vector fields on a compact $n$-dimensional Riemannian manifold with Ricci curvature that is bounded from above. 
\begin{theorem}\label{T6}
Let $(X,g)$ be a compact $n$-dimensional smooth Riemannian manifold with nonzero Killing vector field $\theta^{\sharp}$. There exists a positive constant $C(n,\La)$ such that if $X\in\mathfrak{X}(\lfloor\frac{n}{2}\rfloor,\kappa,D,\La)$ satisfies 
	$$Ric(g)D^{2}(g)\leq C(n,\La),$$
	then the Euler number vanishes, i.e., $\chi(X)=0$.
	\end{theorem}
\begin{remark}
In \cite{CH}, the authors studied certain Dirac operators on a compact $n$-dimensional Riemannian manifold under the assumption that $X$ has either nonzero Euler characteristic or nonzero signature. They proved that for given positive numbers $2p>n$, $\la_{1}$ and $\la_{2}$ if there exists a metric $g$ on $X$ satisfying
$$-\la_{1}\leq Ric(g)\leq\varepsilon,\ diam(g)\leq1, \frac{1}{{\rm{Vol}}(g)}\int_{X}|\mathfrak{R}|^{p}\leq\la_{2}$$
where  $\varepsilon=\varepsilon(p,n,\la_{1},\la_{2})$ is a positive constant, then the isometry group of Riemannian metric $g$ on $X$ is finite. In our conditions stated in Theorem \ref{T6}, the constant $\la_{1}$ obeys $\la_{1}=-(n-1)\kappa$. However, it is not possible to provide a uniform upper bound for $\la_{2}$ with respect to $\kappa,D,n$ and $p$. Thus, our results differ somewhat from theirs.
\end{remark}
Recall that the curvature operator of a smooth Riemannian manifold $(X,g)$ is called $l$-nonnegative if the sum of its lowest $l$ eigenvalues is nonnegative.  Furthermore, if $(X,g)$ is a compact $2n$-dimensional Riemannian manifold with $n$-positive curvature operator, then according to \cite[Theorem A]{PW}, we have $\chi(X)=2$. We then have
\begin{corollary}\label{T7}
	Let $(X,g)$ be a compact $2n$-dimensional smooth Riemannian manifold. Suppose that the curvature operator of $X$ obey one of the following condtions:\\
	(1) the curvature operator is $n$-positive;\\
	(2) the curvature operator is $n$-nonnegative and $X$ has restricted holonomy $SO(2n)$.\\
	If the Ricci curvature satisfies 
	$$Ric(g)D^{2}(g)\leq C(n),$$
	where $C(n)$ is a uniform positive constant,  then the isometry group of Riemannian metric $g$ on $X$ is finite.
\end{corollary}
\begin{remark}
The second condition in Corollary \ref{T7} is sharp. For example, the $2n$-dimensional torus $T^{2n}$ has $\chi(T^{2n})=0$ and admits a flat metric with zero curvature eigenvalues. In this case, the flat metric on $T^{2n}$ has a trivial holonomy group but an infinite isometry group.
\end{remark}
\section{Preliminaries}
	\subsection{Curvature operator}
Let's first recall the definition of the curvature operator and the latest research progress in this area by Petersen and Wink \cite{PW}.	
	
Let $(X,g)$ be a compact $n$-dimensional Riemannian manifold and let $R(X,Y)Z=\na_{Y}\na_{X}Z-\na_{X}\na_{Y}Z+\na_{[X,Y]}Z$ denote its curvature tensor. The Weitzenb\"{o}ck curvature operator acting on a $(0,k)$-tensor $T$ is defined by
$$Ric(T)(X_{1},\cdots,X_{k})=\sum_{i=1}^{k}\sum_{j=1}^{n} (R(X_{i},e_{j})T)(X_{1},\cdots,e_{j},\cdots,X_{k}).$$
	We use a variation of the Ricci tensor to symbolize this as it is the Ricci tensor when evaluated on vector fields and $1$-forms. The Lichnerowicz Laplacian is then given by
	$$\De_{L}T=\na^{\ast}\na T+cRic(T)$$
	for some constant $c>0$. Notably, the Hodge Laplacian on differential forms corresponds to the case $c=1$.
	
	For a $(0,k)$-tensor and  $L\in\mathfrak{so}(TX)$, 	Petersen and Wink defined the action
	$$(LT)(X_{1},\cdots,X_{k})=-\sum_{i=1}^{k}T(X_{1},\cdots,LX_{i},\cdots,X_{k}).$$
	Noting that a $(0,k)$-tensor $T$ can be changed to a tensor $\hat{T}$ with values in $\La^{2}TX$. Implicitly this works as follows
	$$g(L,\hat{T}(X_{1},\cdots,X_{k}))=(LT)(X_{1},\cdots,X_{k}),$$
	for all $L\in\mathfrak{so}(TX)$.
	\begin{proposition}(\cite[Proposition 1.4]{PW})\label{P1}
		Let $R:\La^{2}T_{p}X\rightarrow\La^{2}T_{p}X$ be the curvature operator and $\Xi_{i}$ an othonormal basis for $\La^{2}T_{p}X$. If $\Xi_{i}$ is an eigenbasis of $R$ and $\la_{i}$ denote the corresponding eigenvalues, then
		\begin{equation}
			g(Ric(T),T)=\sum_{i}\la_{i}|\Xi_{i}T|^{2}.
		\end{equation}
	\end{proposition}
Petersen and Wink \cite{PW} observed that for any $p$-form $\a$ ($1\leq p\leq \lfloor\frac{n}{2}\rfloor$),
	$$|\Xi_{i}\a|^{2}\leq p|\a|^{2}$$
	while 
	$$\sum_{i}|\Xi_{i}T|^{2}=p(n-p)|\a|^{2}.$$
They also observed the following lemma which provides a general way to control the curvature term of the Lichnerowicz Laplacian $\De_{L}$ on $(0,k)$-tensors.
	\begin{lemma}(\cite[Lemma 2.1]{PW})\label{L1}
		Let $\mathfrak{R}:\La^{2}V\rightarrow\La^{2}V$ be an algebraic curvature operator with eigenvalues $\la_{1}\leq\cdots\leq\la_{\binom{n}{2}}$ and let $T\in T^{(0,k)}(V )$. Suppose there is $C\geq1$ such that
		$$|LT|^{2}\leq\frac{1}{C}|\hat{T}|^{2}|L|^{2}$$
		for all $L\in \mathfrak{so}(V)$.
		
		Let $\kappa\leq0$.	If $\frac{1}{\lfloor C\rfloor}(\la_{1}+\cdots+\la_{\lfloor C\rfloor})\geq\kappa$, then $g(\mathfrak{R}(\hat{T}),\hat{T}) \geq\kappa|\hat{T}|^{2}$ and if $\la_{1}+\cdots+\la_{\lfloor C\rfloor}>0$, then $g(\mathfrak{R}(\hat{T}),\hat{T})>0$ unless $\hat{T}=0$.
	\end{lemma}	
We then have
	\begin{corollary}\label{C4}
		Let $(X,g)$ be a compact $n$-dimensional smooth Riemannian manifold, and let $\la_{1}\leq\cdots\leq \la_{\binom{n}{2}}$ denote the eigenvalues of the curvature operator of $(X,g)$. Fix $1\leq p\leq\lfloor\frac{n}{2}\rfloor$. If
		$$\frac{\la_{1}+\cdots+\la_{n-p}}{n-p}\geq\kappa.$$
		then for any $\a\in\Om^{k}(X)$, $k\leq p$ or $k\geq n-p$,
		\begin{equation}\label{E9}
		g(Ric(\a),\a)\geq\kappa|\hat{\a}|^{2}=\kappa k(n-k)|\a|^{2}.
		\end{equation}
		 		\end{corollary}
\begin{proof}
By \cite[Lemma 2.2]{PW}, any $k$-form $\a$ satisfies
$$|L\a|^{2}\leq\min\{k,n-k\}|\a|^{2}|L|^{2}$$
for all $L\in\mathfrak{so}(TX)$. Following \cite[Proposition 2.5]{PW}, we also have
$$|\hat{\a}|^{2}=k(n-k)|\a|^{2}.$$
For the  case $k\leq p$, we have
	$$|L\a|^{2}\leq k|\a|^{2}|L|^{2}=\frac{1}{n-k}|\hat{\a}|^{2}|L|^{2}$$
For the case $k\geq n-p$ ,  we have
$$|L\a|^{2}\leq (n-k)|\a|^{2}|L|^{2}=\frac{1}{k}|\hat{\a}|^{2}|L|^{2}.$$ 
Let $\tilde{k}=\max\{k,n-k\}\geq n-p$. Then 
$$|L\a|^{2}\leq\frac{1}{\tilde{k}}|\hat{\a}|^{2}|L|^{2}$$
and the eigenvalues of the curvature operator satisfy 
		$$\frac{1}{\tilde{k}}(\la_{1}+\cdots+\la_{\tilde{k}})\geq \frac{1}{n-p}(\la_{1}+\cdots+\la_{n-p})\geq \kappa.$$ 
Applying Lemma \ref{L1} yields the inequality (\ref{E9}).
	\end{proof}
	\begin{remark}
		Since the Ricci curvature is bounded from below by the sum of the lowest $(n-1)$ eigenvalues of the curvature operator, $Ric(X)\geq(n-1)\kappa$.
	\end{remark}
	\subsection{Poincar\'{e} inequality }
	In this section, we recall the following Poincar\'{e} inequality (see \cite{Ber,Che}). For $1\leq q$, let $p$ satisfies $1\leq p\leq\frac{nq}{n-q}$ and $p<\infty$. Define the Sobolev constant $\Sigma(n,p,q)$ of the canonical unit sphere $S^{n}$ by
	$$\Sigma(n,p,q)=\sup\{\frac{\|f\|_{L^{p}}}{\|df\|_{L^{q}}}:f\in L^{q}_{1}(S^{n}),f\neq0,\int_{S^{n}}f=0\}.$$ 
	If the Ricci curvature of the compact Riemannian manifold $(X,g)$ satisfies 
	$$Ric(g)D^{2}(g)\geq-(n-1)b^{2},$$
	where $D(g)$ is the diameter of $g$, then there exists a positive number $R(b)=\frac{D(g)}{bC(b)}$ (see \cite[Appendix I, Theorem 2]{Ber}), where $C(b)$ is the unique positive root of the equation
	\begin{equation*}
		x\int_{0}^{b}(\cosh t+x\sinh t)^{n-1}dt=\int_{0}^{\pi}\sin^{n-1}tdt,
	\end{equation*}
	such that the following  Poincar\'{e} inequality holds (see \cite[Page 397]{Ber} or \cite[Theorem 5.1]{Che}).
	\begin{theorem}\label{T1}
		Let $(X,g)$ be a compact $n$-dimensional smooth Riemannian manifold  satisfying
		\begin{equation*}
			Ric(g)D^{2}(g)\geq-(n-1)b^{2},
		\end{equation*}
		for some constant $b>0$. Then for each $1\leq p\leq\frac{nq}{n-q}$, $p<\infty$ and $f\in L^{q}_{1}(X)$, we have
		\begin{equation*}
			\begin{split}
				&\|f-\frac{1}{{\rm{Vol}}(g)}\int_{X}f\|_{L^{p}(X)}\leq S_{p,q}\|df\|_{L^{q}(X)},\\
				&\|f\|_{L^{p}(X)}\leq S_{p,q}\|df\|_{L^{q}(X)}+{\rm{Vol}}(g)^{\frac{1}{p}-\frac{1}{q}}\|f\|_{L^{q}(X)},\\
			\end{split}
		\end{equation*} 
		where ${\rm{Vol}}(g)$ is the volume of $(X,g)$, $S(p,q)=(\frac{{\rm{Vol}}(g)}{{\rm{Vol}}(S^{n})})^{\frac{1}{p}-\frac{1}{q} }R(b)\Sigma(n,p,q)$.
	\end{theorem}
	Chen provided the following lower bound estimate for $bC(b)$ as $ b\rightarrow 0$.
	\begin{lemma}(\cite[Lemma 5.6]{Che})\label{L2}
		Let $C(b)$ be the unique positive root of the equation
		\begin{equation*}
			x\int_{0}^{b}(\cosh t+x\sinh t)^{n-1}dt=\int_{0}^{\pi}\sin^{n-1}tdt.
		\end{equation*}
		Then 
		$$\lim\inf_{b\rightarrow 0}b C(b)\geq a_{n}>0$$
		for some constant $a_{n}$ depending only on $n$. In particular, there is a positive constant $c_{1}(n)$ such that if 
	$b\leq c_{1}(n),$
		then  $bC(b)\geq a_{n}$.
	\end{lemma}
	
	\section{Dirac operator $\mathcal{D}_{\theta}$ on Riemannian manifold  }
	
	\subsection{Index of $\mathcal{D}_{\theta}$}
	Let $(X,g)$ be a compact $n$-dimensional Riemannian manifold. We denote by $\Om^{p}(X)$ the spaces of real differential $p$-forms for $0\leq p\leq n$. To simplify notation, we often identify real vector fields with real $1$-forms using the isomorphisms. Given a smooth $1$-form $\theta\in\Om^{1}(X)$, the vector field $\theta^{\sharp}$ on $X$ is characterized by the condition $\theta(V)=g(V,\theta^{\sharp})$, for all vector filed $V$ on $X$. We denote the contraction by the vector field $\theta^{\sharp}$ as $i_{\theta^{\sharp}}$. The Hodge star operator with respect to metric $g$ is denoted by $\ast$. We define the differential operator $d_{\theta}:\Om^{k}(X)\rightarrow\Om^{k+1}(X)$ as
	\begin{equation*}
	d_{\theta}=d+\theta\wedge.
	\end{equation*}
Its formal adjoint $d^{\ast}_{\theta}:\Om^{k}(X)\rightarrow\Om^{k-1}(X)$ is given by
$$d_{\theta}^{\ast}=d^{\ast}+i_{\theta^{\sharp}}.$$	
\begin{remark}
It is important to note that the differential $d_{\theta}$ does not satisfy the crucial property of $(d_{\theta})^{2}=0$ required for introducing cohomology groups (see Section \ref{S1}). In fact, we have $(d_{\theta})^{2}=d\theta$.	
\end{remark}
Define the twisted Dirac operator $\mathcal{D}_{\theta}$ acting on differential forms by $$\mathcal{D}_{\theta}=d_{\theta}+d^{\ast}_{\theta}:\Om^{+}(X)\rightarrow\Om^{-}(X),$$
where $\Om^{+}(X)=\oplus_{k=even}\Om^{k}(X)$ and $\Om^{-}(X)=\oplus_{k=odd}\Om^{k}(X)$.	
Recall that for the  operator $d+d^{\ast}$, the index computes the Euler number: \begin{equation*}
	\begin{split}
	{\rm{Index}}(d+d^{\ast})&=\dim\ker(d+d^{\ast})-\dim{\rm{coker}}(d+d^{\ast})\\
	&=\dim\bigoplus_{p=even}\mathcal{H}^{k}(X)-\dim\bigoplus_{p=odd}\mathcal{H}^{k}(X)\\
	&=\sum_{p=0}^{n}(-1)^{k}b_{k}(X)\\
	&=\chi(X).	
	\end{split}
	\end{equation*}
	We then have
	\begin{lemma}
		The index of $\mathcal{D}_{\theta}$ satisfies
		$${\rm{Index}}(\mathcal{D}_{\theta}):=\dim\ker(\mathcal{D}_{\theta})-\dim{\rm{coker}}(\mathcal{D}_{\theta})=\chi(X).$$
		Moreover, if $d\theta=0$, then
		$$	\sum_{p=0}^{n}(-1)^{p}\dim\ker d_{\theta}\cap\ker d^{\ast}_{\theta}\cap\Om^{p}(X)=\chi(X).$$
	\end{lemma}
	\begin{proof}
 Since $\mathcal{D}_{\theta}$ is elliptic, its index depends only on the principal symbol, which coincides with that of $d+d^{\ast}$. Thus ${\rm{Index}}(\mathcal{D}_{\theta})=\chi(X).$	
		
When $(d_{\theta})^{2}=0$, we obtain
		$$\ker\mathcal{D}_{\theta}\cap\Om^{\pm}(X)=\bigoplus_{p=even/odd}\ker d_{\theta}\cap\ker d^{\ast}_{\theta}\cap\Om^{p}(X).$$
		Therefore
		\begin{equation*}
			\begin{split}
				{\rm{Index}}(\mathcal{D}_{\theta})&=\sum_{p=even}\dim\ker d_{\theta}\cap\ker d^{\ast}_{\theta}\cap\Om^{p}(X)-\sum_{p=odd}\dim\ker d_{\theta}\cap\ker d^{\ast}_{\theta}\cap\Om^{p}(X)\\
				&=\sum_{p=0}^{n}(-1)^{p}\dim\ker d_{\theta}\cap\ker d^{\ast}_{\theta}\cap\Om^{p}(X).
			\end{split}
		\end{equation*}
	\end{proof}
	\subsection{Integral formula}	
	We begin by establishing a fundamental integral identity that plays a central role in our analysis. In \cite{Che,CH}, the authors also derived some similar integral equalities  for some certain Dirac operators, but our proof method differs from their. 
	\begin{lemma}(\cite{EFM})\label{L5}
		For any $u,v\in\Om^{p}(X)$, we have a pointwise identity as follows:
		\begin{equation}
		\int_{X}\langle L_{\theta^{\sharp}}u,v\rangle=\int_{X}\langle i_{\theta^{\sharp}}u,d^{\ast}v\rangle+\int_{X}\langle du,\theta\wedge v\rangle=\int_{X}\langle\na_{\theta^{\sharp}}u,v\rangle+\int_{X}p\langle(\na \theta^{\sharp})u,v\rangle,
		\end{equation}
		where $L_{\theta^{\sharp}}:=di_{\theta^{\sharp}}+i_{\theta^{\sharp}}d$ is the Lie derivative in the direction $\theta^{\sharp}$ and 
		$$[(\na\theta^{\sharp})u](e_{1},\cdots,e_{p})=\frac{1}{p}\sum_{i}u(e_{1},\cdots,\na_{e_{i}}\theta^{\sharp},\cdots,e_{p}).$$
	\end{lemma}

	\begin{theorem}(cf. \cite[Theorem 4.1 and Corllary 4.3]{Che})\label{T3}
		Let $(X,g)$ be a compact $n$-dimensional smooth Riemannian manifold and $\theta^{\sharp}$ a smooth vector field on $X$. Then for each $\a\in\ker\mathcal{D}_{\theta}\cap\Om^{\pm}(X)$, there is a positive constant $C_{1}(n)$ such that 
		\begin{equation*}
			\int_{X}|\theta^{\sharp}|^{2}|\a|^{2}\leq C_{1}(n)\int_{X}|\na \theta^{\sharp}|\cdot|\a|^{2}.
		\end{equation*}	
	\end{theorem}
	\begin{proof}
		We only prove the case of $n=even$. Let $\a=\sum_{k=0}^{n}\a_{2k}\in\Om^{+}(X)$ be a smooth differential form on $X$, where $\a_{2k}\in\Om^{2k}(X)$. The condition $\a\in\ker\mathcal{D}_{\theta}$ implies
		$$(d+d^{\ast})\a=-\theta\wedge\a-i_{\theta^{\sharp}}\a.$$
		We then have
		\begin{equation*}
		\begin{split}
		\int_{X}\langle L_{\theta^{\sharp}}\a,\a\rangle&=
		\int_{X}\langle i_{\theta^{\sharp}}\a,d^{\ast}\a\rangle+\int_{X}\langle d\a,\theta\wedge \a\rangle\\
		&=\int_{X}\langle i_{\theta^{\sharp}}\a,-d\a-\theta\wedge\a-i_{\theta^{\sharp}}\a\rangle+\int_{X}\langle -d^{\ast}\a-\theta\wedge\a-i_{\theta^{\sharp}}\a,\theta\wedge \a\rangle\\
		&=-\|i_{\theta^{\sharp}}\a\|^{2}-\int_{X}\langle \a,\theta\wedge d\a\rangle-\|\theta\wedge\a\|^{2}-\int_{X}\langle \a,d(\theta\wedge \a)\rangle\\
	&=-\|i_{\theta^{\sharp}}\a\|^{2}-\|\theta\wedge\a\|^{2}-\int_{X}\langle \a,d\theta\wedge \a\rangle\\
	&=-\int_{X}|\theta^{\sharp}|^{2}|\a|^{2}-\int_{X}\langle \a,d\theta\wedge \a\rangle,\\
		\end{split}
		\end{equation*}
	where we used the identity  
	$$|\theta\wedge\b|^{2}+|i_{\theta^{\sharp}}\b|^{2}=|\theta^{\sharp}|^{2}|\b|^{2},\ \forall\b\in\Om^{p}(X).$$
	On the other hand, following Lemma \ref{L5}, we have
		\begin{equation*}
		\begin{split}
		\int_{X}\langle L_{\theta^{\sharp}}\a,\a\rangle&=\sum_{k=0}^{n}
		\int_{X}\langle L_{\theta^{\sharp}}\a_{2k},\a_{2k}\rangle\\
		&=\sum_{k=0}^{n}\int_{X}\langle i_{\theta^{\sharp}}\a_{2k},d^{\ast}\a_{2k}\rangle+\sum_{k=0}^{n}\int_{X}\langle d\a_{2k},\theta\wedge \a_{2k}\rangle\\
		&=\sum_{k=0}^{n}\int_{X}\langle\na_{\theta^{\sharp}}\a_{2k},\a_{2k}\rangle+\sum_{k=0}^{n}\int_{X}2k\langle(\na \theta^{\sharp})\a_{2k},\a_{2k}\rangle\\
		\end{split}
		\end{equation*}
	Given $u,v\in\Om^{\bullet}(X)$ we have (see \cite[Example 2.1.13]{Nic})
	$$L_{\theta^{\sharp}}\langle u,v\rangle=\langle\na_{\theta^{\sharp}}u,v\rangle+\langle u,\na_{\theta^{\sharp}}v\rangle.$$
	We set $u=v$. Integrating over $X$ and using the divergence formula  we deduce
	$$2\int_{X}\langle\na_{\theta^{\sharp}}u,u\rangle=\int_{X}1\cdot L_{\theta^{\sharp}}\langle u,u\rangle=-\int_{X}({\rm{div}}\theta^{\sharp})\langle u,u\rangle.$$	
		 Combining above identities yield that
		\begin{equation*}
			\begin{split}
				\int_{X}|\theta^{\sharp}|^{2}|\a|^{2}&=-\sum_{k=0}^{n}\int_{X}\langle\na_{\theta^{\sharp}}\a_{2k},\a_{2k}\rangle+\sum_{k=0}^{n}\int_{X}2k\langle(\na \theta^{\sharp})\a_{2k},\a_{2k}\rangle-\int_{X}\langle \a,d\theta\wedge \a\rangle \\
				&\leq C_{1}(n)\int_{X}|\na\theta^{\sharp}|\cdot|\a|^{2}.\\
			\end{split}
		\end{equation*}	
	\end{proof}	
		\subsection{Priori estimates of smooth forms on $\ker\mathcal{D}_{\theta}$}
We establish an a priori $L^{p}$-estimate for the smooth differential $\a\in\ker\mathcal{D}_{\theta}\cap\Om^{\pm}(X)$ via Morse iteration. The analysis begins with a fundamental pointwise identity (\cite[Page 181]{Pet}),
\begin{proposition}
For any $\a\in\Om^{p}(X)$, the following identity holds:
\begin{equation}\label{E8}
\begin{split}
-\frac{1}{2}\De_{d}|\a|^{2}&=|\na\a|^{2}-\langle\na^{\ast}\na\a,\a\rangle\\
	&=|\na\a|^{2}-\langle\De_{d}\a,\a\rangle+\langle Ric(\a),\a\rangle,\\
\end{split}	
\end{equation}
\end{proposition}
\begin{remark}
Note that $\langle\De_{d}\a,\a\rangle$ is not always equal to $|d\a|^{2}+|d^{\ast}\a|^{2}$ pointwise. For more details, please refer to \cite{CH}.	
\end{remark}
\begin{proposition}(cf. \cite{CH})
		Let $(X,g)$ be a compact $n$-dimensional smooth Riemannian manifold in  $\mathfrak{X}(\lfloor\frac{n}{2}\rfloor,\kappa,D,\La)$, $\theta^{\sharp}$ be a smooth vector field  on $X$. For each $\a\in\ker\mathcal{D}_{\theta}\cap\Om^{\pm}(X)$, we have
		\begin{equation}\label{E22}
	-\frac{1}{2}\De_{d}|\a|^{2}\geq|\na\a|^{2}-|\theta^{\sharp}|^{2}|\a|^{2}+c(n)\kappa|\a|^{2}-{\rm{div}}V,
	\end{equation}
	where $c(n)=\lfloor\frac{n}{2}\rfloor(n-\lfloor\frac{n}{2}\rfloor)$.
	\end{proposition}
	\begin{proof}
Since the operator $d+d^{\ast}:\Om^{+}(X)\rightarrow \Om^{-}(X)$ is a Dirac operator (see \cite{Law}) and $\De_{d}=(d+d^{\ast})^{2}$, we can apply a key observation from \cite{CH} (also see  \cite[Page 115, Proposition 5.3]{Law}). For any smooth differential form $\a$, we have
		\begin{equation*}
		\langle\De_{d}\a,\a\rangle=\langle(d+d^{\ast})\a,(d+d^{\ast})\a\rangle+{\rm{div}}V,
		\end{equation*}	
		where $V$ is a vector field defined by the condition that
		\begin{equation*}
		\langle V,W\rangle=-\langle (d+d^{\ast})\a,W^{\sharp}\wedge\a+i_{W}\a\rangle.
		\end{equation*}			
	Following (\ref{E8}), we have
		\begin{equation*}
		-\frac{1}{2}\De_{d}|\a|^{2}=|\na\a|^{2}-|(d+d^{\ast})\a|^{2}+g (Ric(\a),\a)-{\rm{div}}V.\\
		\end{equation*}	
	According to Corollary \ref{C4}, for any $\b\in\Om^{k}(X)$, $0\leq k\leq n$, we have
		$$g(Ric(\b),\b)\geq\kappa k(n-k)|\b|^{2}.$$
		Therefore, 
		$$g (Ric(\a),\a)\geq\min_{0\leq k\leq n}{\kappa k(n-k)}|\a|^{2}=\kappa\lfloor\frac{n}{2}\rfloor(n-\lfloor\frac{n}{2}\rfloor)|\a|^{2}.$$
	By $\mathcal{D}_{\theta}\a=0$, it implies that 	$(d+d^{\ast})\a=-\theta\wedge\a-\textit{i}_{\theta^{\sharp}}\a$. Therefore,
		\begin{equation}\label{E3}
		|V|\leq C_{2}(n)|\theta^{\sharp}|\cdot|\a|^{2},
		\end{equation}	
	where $C_{2}(n)$ is a positive constant only dependent on $n$. Combining above inequalities with $$|(d+d^{\ast})\a|^{2}=|\theta^{\sharp}|^{2}|\a|^{2},$$
	we obtain
	\begin{equation*}
		-\frac{1}{2}\De_{d}|\a|^{2}\geq|\na\a|^{2}-|\theta^{\sharp}|^{2}|\a|^{2}+c(n)\kappa|\a|^{2}-{\rm{div}}V.
		\end{equation*}
	\end{proof}	
	\begin{theorem}\label{T8}
		Let $(X,g)$ be a compact $n$-dimensional smooth Riemannian manifold in $\mathfrak{X}(\lfloor\frac{n}{2}\rfloor,k,D,\La)$, $\theta^{\sharp}$ be a smooth vector field  on $X$ . For each $\a\in\ker\mathcal{D}_{\theta}\cap\Om^{\pm}(X)$, we then have
		\begin{equation}\label{E17}
			\|\a\|_{L^{2\nu^{N}}(X)}\leq(\frac{\int_{X}|\a|^{2}}{{\rm{Vol}}(g)})^{\frac{1}{2}}\exp(NC_{3}(n)R(\La)\sqrt{\la})({\rm{Vol}}(g))^{\frac{1}{2\nu^{N} }},\ \forall N\in\mathbb{N}^{+}\\
		\end{equation}
		where $\la:=2C_{2}^{2}(n)\max|\theta^{\sharp}|^{2}-c(n)\kappa$, $\nu=\frac{n}{n-2}$, and $C_{2}(n)$, $C_{3}(n)$ are  uniform positive constants.
	\end{theorem}
	\begin{proof}
Following the Second Kato inequality (see \cite{Ber}), we  have
		\begin{equation*}
		\begin{split}
		-|\a|\De_{d}|\a|&\geq-\langle\na^{\ast}\na\a,\a\rangle\\
		&=-\langle\De_{d}\a,\a\rangle+g(Ric(\a),\a).\\
		\end{split}
		\end{equation*}
Using (\ref{E22}), we obtain
		\begin{equation}\label{E24}
		\begin{split}
		-|\a|\De_{d}|\a|&\geq g(Ric(\a),\a)-|\theta^{\sharp}|^{2}|\a|^{2}-{\rm{div}}V\\
		&\geq-|\theta^{\sharp}|^{2}|\a|^{2}+c(n)\kappa|\a|^{2}-{\rm{div}}V,\\
		\end{split}
		\end{equation}	
		where $c(n)$ is an uniform positive constant. For any $k\geq1$, we multiply the inequality (\ref{E24}) by $f^{2k-2}$, where $f=|\a|$, and integrate. Then we obtain
		\begin{equation}\label{E4}
			\begin{split}
			\int_{X}f^{2k-1}\De_{d}f
				&\leq -c(n)\kappa\int_{X}f^{2k}\\
				&+\int_{X}|\theta^{\sharp}|^{2}f^{2k}\\
				&+|({\rm{div}}V,f^{2k-2})_{L^{2}(X)}|\\
			\end{split}
		\end{equation}
		Noting that 
		\begin{equation}\label{E7}
			\begin{split}
		|({\rm{div}}V,f^{2k-2})_{L^{2}(X)}|&=|(V,\na f^{2k-2})_{L^{2}(X)}|\\
				&= (2k-2)|(V,f^{2k-3}\na f)_{L^{2}(X)}|\\
				&\leq (2k-2)C_{2}(n)\int_{X}f^{2k-1}|\theta^{\sharp}||\na f|\ (by\ (\ref{E3}))\\
				&\leq \frac{(2k-2)}{4}\int_{X}f^{2k-2}|\na f|^{2}+(2k-2)C_{2}^{2}(n)\int_{X}f^{2k}|\theta^{\sharp}|^{2}.
			\end{split}
		\end{equation}
	Here we use the fact (see \cite{Pet})
	$${\rm{div}}(h\cdot V)=g(\na h,V)+h\cdot{\rm{div}}V$$
	and $$\int_{X}{\rm{div}}(h\cdot V)=0$$
	for any function $h$. Combining inequalities (\ref{E4}) and (\ref{E7}) yields
		\begin{equation*}
			\begin{split}
				(2k-1)\int_{X}f^{2k-2}|\na f|^{2}&=\int_{X}f^{2k-1}\De_{d}f\\
				&\leq  -c(n)\kappa\int_{X}f^{2k}\\
				&\quad+\frac{(2k-2)}{4}\int_{X}f^{2k-2}|\na f|^{2}\\
				&\quad+((2k-2)C_{2}^{2}(n)+1)\int_{X}f^{2k}|\theta^{\sharp}|^{2}.\\
			\end{split}
		\end{equation*}
		We then obtain
		\begin{equation}\label{E23}
		\begin{split}
			\int_{X}f^{2k-2}|\na f|^{2}
			&\leq\frac{2 }{3k-1}((2k-2)C_{2}^{2}(n)+1)\int_{X}f^{2k}|\theta^{\sharp}|^{2}-\frac{2 }{3k-1}c(n)\kappa\int_{X}f^{2k} \\
			&\leq\big{(}2C_{2}^{2}(n)\max|
			\theta^{\sharp}|^{2}-c(n)\kappa\big{)}\int_{X}f^{2k}.\\
			\end{split}
		\end{equation}
		Therefore,
		\begin{equation*}
			\begin{split}
				\int_{X}|df^{k}|^{2}&=k^{2}\int_{X}f^{2k-2}|\na f|^{2}\\
				&\leq k^{2}\big{(}2C_{2}^{2}(n)\max|
				\theta^{\sharp}|^{2}-c(n)\kappa\big{)}\int_{X}f^{2k}.\\
			\end{split}
		\end{equation*}
		We denote $\la:=2C_{2}^{2}(n)\max|\theta^{\sharp}|^{2}-c(n)\kappa$. Let $q=2$, $p=\frac{2n}{n-2}$, $\nu=\frac{n}{n-2}$. By Sobolev inequality in Theorem \ref{T1} and (\ref{E23}), we obtain
		\begin{equation*}
			\begin{split}
				\|f^{k}\|_{L^{2\nu}(X)}&\leq{\rm{Vol}}^{-\frac{1}{n}}(g)(C_{3}(n)R(\La)\|df^{k}\|_{L^{2}(X)}+\|f^{k}\|_{L^{2}(X)})\\
				&\leq{\rm{Vol}}^{-\frac{1}{n}}(g)(C_{3}(n)R(\La)\sqrt{k^{2}\la}\|f\|^{k}_{L^{2k}(X)}+\|f\|^{k}_{L^{2k}(X)}),\\
			\end{split}
		\end{equation*}
		where $C_{3}(n)=\Sigma(n,\frac{n}{n-2},2)[{\rm{Vol}}(S^{n})]^{\frac{1}{n}}$,	and then
		\begin{equation*}
			\|f\|_{L^{2\nu k}(X)}\leq{\rm{Vol}}^{-\frac{1}{nk}}(g)(C_{3}(n)R(\La)\sqrt{k^{2}\la}+1)^{\frac{1}{k}}\|f\|_{L^{2k}(X)}.
		\end{equation*}
		Now we set $k(l)=\nu^{l}$ of each $l\in\mathbb{N}$. By iterating from $l=0$, we can obtain
		\begin{equation*}
			\begin{split}
				\|f\|_{L^{2k(N)}(X)}&\leq\Pi_{l=0}^{N-1}{\rm{Vol}}^{-\frac{1}{nk}}(g)(C_{3}(n)R(\La)\sqrt{k^{2}\la}+1)^{\frac{1}{k}}\|f\|_{L^{2}(X)}\\
				&\leq(\frac{\int_{X}f^{2}}{{\rm{Vol}}(g)})^{\frac{1}{2}}\exp(NC_{3}(n)R(\La)\sqrt{\la})({\rm{Vol}}(g))^{\frac{1}{2\nu^{N} }}\\
			\end{split}
		\end{equation*}
		The  product can be estimated by taking logarithms and using $\log(1+x)\leq x$.
		\begin{equation*}
			\begin{split}
				\Pi_{l=0}^{N-1}(C_{3}(n)R(\La)\sqrt{k^{2}\la}+1)^{\frac{1}{k}}&\leq\exp\sum_{l=0}^{N-1}\frac{1}{k}\log(C_{3}(n)R(\La)\sqrt{k^{2}\la}+1) \\
				&\leq\exp\sum_{l=0}^{N-1}C_{3}(n)R(\La)\sqrt{\la}\\
				&=\exp(NC_{3}(n)R(\La)\sqrt{\la})\\
			\end{split}
		\end{equation*}
	\end{proof}	
Noting that 
$$Ric(g)D^{2}\geq(n-1)\kappa D^{2}\geq -(n-1)\La^{2}$$
for any manifold $X$ belongs to $\mathfrak{X}(p,\kappa,D,\La)$. Let $p=2$, $q=2$ in Theorem \ref{T1}, then we get the following  Poincar\'{e} inequality.
	\begin{corollary}\label{C5}
		Let $(X,g)$ be a compact  $n$-dimensional smooth Riemannian manifold. If  $X\in\mathfrak{X}(p,\kappa,D,\La)$, then for each $f\in L^{2}_{1}(X)$, we have
		\begin{equation*}
			\|f-\frac{1}{{\rm{Vol}}(g)}\int_{X}f\|_{L^{2}(X)}\leq C_{4}(n)R(\La)\|df\|_{L^{2}(X)}
		\end{equation*} 
		where $C_{4}(n)=\Sigma(n,2,2)$ is a positive constant. 
	\end{corollary}
Using the integral formula for smooth differential forms $\a\in\mathcal\ker{\mathcal{D}}_{\theta}\cap\Om^{\pm}(X)$, as established in Theorem \ref{T3}, we have the following lemma 
\begin{lemma}(cf. \cite[Lemma 5.5]{Che})\label{L7}
	Let $(X,g)$ be a compact  $n$-dimensional smooth Riemannian manifold in  $\mathfrak{X}(\lfloor\frac{n}{2}\rfloor,\kappa,D,\La)$, let $\theta^{\sharp}$ be a vector field on $X$. Then for each $\a\in\mathcal\ker{\mathcal{D}}_{\theta}\cap\Om^{\pm}(X)$,  we have
	\begin{equation}\label{E25}
	\int_{X}|f-\bar{f}||\a|^{2}
\leq 2C_{4}(n)R(\La)(\frac{\int_{X}|\a|^{2}}{{\rm{Vol}}(g)})\exp(2NC_{3}(n)R(\La)\sqrt{\la})({\rm{Vol}}(g))^{\frac{1}{2 }} \|\theta^{\sharp}\|_{L^{\infty}(X)}\|\na \theta^{\sharp}\|_{L^{2}(X)},
	\end{equation}
 and
	\begin{equation}\label{E26}
	\int_{X}|\theta^{\sharp}|^{2}|\a|^{2}
	\leq C_{1}(n)(\int_{X}|\na\theta^{\sharp}|^{2})^{\frac{1}{2}}(\frac{\int_{X}|\a|^{2}}{{\rm{Vol}}(g)})\exp(2NC_{3}(n)R(\La)\sqrt{\la})({\rm{Vol}}(g))^{\frac{1}{2}},
	\end{equation}	
where $f=|\theta^{\sharp}|^{2}$, $\bar{f}=\frac{\int_{X}|\theta^{\sharp}|^{2} }{{\rm{Vol}}(g)}$, $N$ is an integer such that $2\nu^{N-1}<4\leq2\nu^{N}$, $\nu=\frac{n}{n-2}$.
\end{lemma}
\begin{proof}
	By Corollary \ref{C5}, we obtain
	\begin{equation*}
	\begin{split}
	\int_{X}|f-\bar{f}||\a|^{2}&\leq\|\a\|^{2}_{L^{4}(X)}(\int_{X}|f-\bar{f}|^{2})^{\frac{1}{2}}\\
	&\leq C_{4}(n)R(\La)\|\a\|^{2}_{L^{4}(X)}(\int_{X}|\na f|^{2})^{\frac{1}{2}}\\
	&=2C_{4}(n)R(\La)\|\a\|^{2}_{L^{4}(X)}(\int_{X}|\theta^{\sharp}|^{2}\cdot|\na |\theta^{\sharp}||^{2})^{\frac{1}{2}}\\
	&\leq 2C_{4}(n)R(\La)\|\a\|^{2}_{L^{4}(X)}\|\theta^{\sharp}\|_{L^{\infty}(X)}\|\na \theta^{\sharp}\|_{L^{2}(X)}.\\
	\end{split}
	\end{equation*}
	Here we use the Kato inequality $|\na|\theta^{\sharp}||\leq|\na\theta^{\sharp}|$. Following the integral formula in Theorem \ref{T3}, we have
	\begin{equation*}
	\begin{split}
	\int_{X}|\theta^{\sharp}|^{2}|\a|^{2}&\leq C_{1}(n)\int_{X}|\na \theta^{\sharp}||\a|^{2}\\
	&\leq C_{1}(n)(\int_{X}|\na \theta^{\sharp}|^{2})^{\frac{1}{2}}(\int_{X}|\a|^{4})^{\frac{1}{2}}.\\
	\end{split}
	\end{equation*}
We can choose an integer $N=N(n)$ such $2\nu^{N-1}<4\leq2\nu^{N}$, where $\nu=\frac{n}{n-2}$. Following (\ref{E17}), we have
\begin{equation*}
\begin{split}
\|\a\|_{L^{4}(X)}&\leq\|\a\|_{L^{2v^{N}}(X)}({\rm{Vol}}(g))^{\frac{1}{4}-\frac{1}{2\nu^{N} }}\\
&\leq(\frac{\int_{X}|\a|^{2}}{{\rm{Vol}}(g)})^{\frac{1}{2}}\exp(NC_{3}(n)R(\La)\sqrt{\la})({\rm{Vol}}(g))^{\frac{1}{4 }}.\\
\end{split}
\end{equation*}
Combining above inequalities, we have
\begin{equation*}
\begin{split}
\int_{X}|f-\bar{f}||\a|^{2}&\leq 2C_{4}(n)R(\La)\|\a\|^{2}_{L^{4}(X)}\|\theta^{\sharp}\|_{L^{\infty}(X)}\|\na \theta^{\sharp}\|_{L^{2}(X)}\\
&\leq 2C_{4}(n)R(\La)(\frac{\int_{X}|\a|^{2}}{{\rm{Vol}}(g)})\exp(2NC_{3}(n)R(\La)\sqrt{\la})({\rm{Vol}}(g))^{\frac{1}{2 }} \|\theta^{\sharp}\|_{L^{\infty}(X)}\|\na \theta^{\sharp}\|_{L^{2}(X)},\\	
\end{split}
\end{equation*}
and
\begin{equation*}
\begin{split}
\int_{X}|\theta^{\sharp}|^{2}|\a|^{2}&\leq C_{1}(n)(\int_{X}|\na\theta^{\sharp}|^{2})^{\frac{1}{2}}(\int_{X}|\a|^{4})^{\frac{1}{2}}\\
&\leq C_{1}(n)(\int_{X}|\na\theta^{\sharp}|^{2})^{\frac{1}{2}}(\frac{\int_{X}|\a|^{2}}{{\rm{Vol}}(g)})\exp(2NC_{3}(n)R(\La)\sqrt{\la})({\rm{Vol}}(g))^{\frac{1}{2}}.\\
\end{split}
\end{equation*}	
\end{proof}
\section{Curvature operator with small lower bounded }
Let  $\theta$ be a smooth closed $1$-form on a compact manifold $X$. The closedness of $\theta$ implies $d_{\theta}^{2}=0$, which means that $d_{\theta}$ gives rise to a cohomology, denoted by $H^{\ast}(X,\theta)$, known as the Morse-Novikov cohomology. When $[\theta]=0\in H^{1}_{dR}(X)$, the Morse-Novikov cohomology group $H^{p}(X,\theta)$ is isomorphic to the de Rham cohomology group $H^{p}_{dR}(X)$.
	While Morse-Novikov cohomology depends on the choice of $\theta$, the Euler number of the Morse-Novikov cohomology is equal to the Euler number $\chi(X)$ of the manifold. The Morse-Novikov cohomology also plays a significant role in the study of manifold topology. Although the relationship between curvature and de Rham cohomology is classical, the interplay between curvature and Morse-Novikov cohomology remains less explored.
	
Pioneering work by Chen \cite{Che,CJH,CH} established integral formulas and estimates for  $\De_{\theta}$-harmonic forms. In this section, we extend Chen's ideas to study the topology of manifolds $X\in\mathfrak{X}(p,\kappa,D,\La)$. We establish a vanishing result for certain Morse-Novikov cohomology groups with respect to $t\theta$.
	\begin{theorem}\label{T4}
		Let $(X,g)$ be a compact $n$-dimensional, $(n\geq3)$, smooth Riemannian manifold with nonzero first de Rham cohomology group.  There exists a positive constant $C(n,\La)$ such that if $X\in\mathfrak{X}(p,\kappa,D,\La)$ satisfies 
		$$Ric(g)D^{2}(g)\geq -C(n,\La),$$
		then there exists some $t\in\mathbb{R}$, $t\neq0$ such that $H^{k}(X,t\theta)=0$ for all $k\leq p$, where $H^{k}(X,t\theta)$ is the Morse-Novikov $p$-th cohomology group with respect to $t\theta$. 
	\end{theorem}
We can establish a vanishing result for certain Morse-Novikov cohomology groups of a compact $n$-dimensional smooth Riemannian manifold $X\in\mathfrak{X}(p,\kappa,D,\La)$ with nonzero first de Rham cohomology group when $\La$ is sufficiently small.  
\begin{corollary}\label{T5}
Let $(X,g)$ be a compact $n$-dimensional smooth Riemannian manifold with nonzero first de Rham cohomology group. There exists a uniform positive constant $C(n)$ such that if $X\in\mathfrak{X}(p,\kappa,D,\La)$ satisfies 
$$\La\leq C(n),$$
then there exists some $t\in\mathbb{R}$, $t\neq0$ such that $H^{k}(X,t\theta)=0$ for all $k\leq p$, where $H^{k}(X,t\theta)$ is the Morse-Novikov $p$-th cohomology group with respect to $t\theta$. 
\end{corollary}
We will temporarily postpone the proof of Theorem \ref{T4} and instead prove Theorem \ref{C7} using this conclusion. The detailed proof of Theorem \ref{T4} will be provided in the following section.
\begin{proof}[\textbf{Proof of Theorem \ref{C7}}]
By Poincar\'{e} duality for Morse-Novikov cohomology $H^{p}(X,t\theta)\cong H^{2n-p}(X,t\theta)$, the Euler number obeys
$$\chi(X)=\sum_{p=0}^{2n}(-1)^{p}\dim H^{p}(X,t\theta)=\sum_{p< n}2(-1)^{p}\dim H^{p}(X,t\theta)+(-1)^{n}\dim H^{n}(X,t\theta)$$
Following Corollary \ref{T5}, we can immediately conclude that 
$\chi(X)=0$ when $k=n$ (since $H^{p}(X,t\theta)=0$ for any $p$)  and $(-1)^{n}\chi(X)=\dim H^{n}(X,t\theta)\geq0$ (since $H^{p}(X,t\theta)=0$ for all $p\neq n$).
\end{proof}
\begin{proof}[\textbf{Proof of Corollary \ref{C1}}]
Let's consider the case where $\dim X=4$. The Euler number of $X$ satisfies  $$\chi(X)=2+b_{2}(X)-2b_{1}(X).$$ 
If $b_{1}(X)=0$, then we get $\chi(X)\geq2$. \\
If $b_{1}(X)\geq1$, i.e., there exists a non-zero harmonic $1$-form $\theta$, and some $t\in\mathbb{R}$, $t\neq0$ such that $H^{k}(X,t\theta)=0$ for $k=0,1$ (see Corollary \ref{T5}). Consequently, we have $\chi(X)=\dim H^{2}(X,t\theta)\geq0$.\\
Thus $\chi(X)\geq0$ in all cases.
\end{proof} 
	\subsection{Morse-Novikov cohomology}\label{S1}
	Let $X$ be a compact smooth manifold equipped with a closed $1$-form $\theta$.
	The Morse-Novikov cohomology (independently introduced by Lichnerowicz \cite{Lic}, Novikov \cite{Nov}, and Guedira-Lichnerowicz \cite{GL}) is defined via the complex  $(\Om^{\ast}(X),d_{\theta})$
	\begin{equation}\label{E10}
		\Om^{0}(X)\xrightarrow{d_{\theta}}\Om^{1}(X)\xrightarrow{d_{\theta}}\Om^{2}(X)\xrightarrow{}\cdots
	\end{equation}
	Denote by $H^{\ast}(X,\theta)$ the cohomology of the complex $(\Om^{\ast}(X),d_{\theta})$. 
	\begin{proposition}(\cite[Proposition 4.4]{LLMP} and \cite{GL})
		Let $X$ be a compact smooth manifold and $\theta$ a closed $1$-form on $X$. Then, \\
		(i) The differential complex $(\Om^{\ast}(M),d_{\theta})$ is elliptic and the cohomology groups $H^{k}(X,\theta)$ have finite dimension.\\
		(ii) If $\theta$ is exact, then $H^{k}(X,\theta)\cong H^{k}_{dR}(X)$.\\
		(iii) If $\theta$ is not exact, and $X$ is connected and orientable, then $H^{0}(X,\theta)$ and $H^{n}(X,\theta)$ vanish.
	\end{proposition}
	In general, if the $1$-form $\theta$ is not exact then $H^{k}(X,\theta)\ncong H^{k}_{dR}(X)$.  Since the complex $(\Om^{\ast}(X),d_{\theta})$ is elliptic, we obtain an orthogonal decomposition of the space $\Om^{k}(X)$ as follows
	\begin{equation}\label{E11}
		\Om^{k}(X)=\mathcal{H}^{k}_{\theta}(X)\oplus d_{\theta}(\Om^{k-1}(X))\oplus d^{\ast}_{\theta}(\Om^{k+1}(X)),
	\end{equation}
	where  $$\mathcal{H}_{\theta}^{k}(X)=\{\a\in\Om^{k}(X):d_{\theta}(\a)=0, d^{\ast}_{\theta}(\a)=0 \}.$$
	From (\ref{E11}), it follows that $H^{k}(X,\theta)\cong\mathcal{H}^{k}_{\theta}(X)$.  But Morse-Novikov cohomology $H^{k}(X,\theta)$ is different to de Rham cohomology, it is not a topological invariant, it depends on $[\theta]\in H^{1}_{dR}(M)$. Following the Atyiah-Singer index theorem, it is known  that the index of the elliptic complex $(\Om^{\ast}(X),d_{\theta})$ is independent of $\theta$. Moreover, the Euler number of the Morse-Novikov cohomology coincides with the Euler number of the manifold, given by
	\begin{equation}\label{E12}
		\sum_{p=0}^{n}(-1)^{p}\dim\mathcal{H}^{p}_{\theta}(X)=\chi(X),
	\end{equation}
	where $\chi(X)$ is the Euler number of $X$. This result highlights the topological significance of the Morse-Novikov cohomology. 
	\subsection{Harmonic $1$-forms and Euler number}
	For $n\geq3$, define the constant $\sigma_{n}$ by
	$$\sigma_{n}=\Sigma(n,\frac{2n}{n-2},2)[{\rm{Vol}}(S^{n})]^{\frac{1}{n}}.$$
	Define the function $B:\mathbb{R}^{+}\rightarrow\mathbb{R}^{+}$ by
	$$B(n,\nu,x)=\Pi_{i=0}^{\infty}(x\nu^{i}(2\nu^{i}-1)^{-\frac{1}{2}}+1)^{2\nu^{-i}},\ \nu=\frac{n}{n-2}.$$
	The function $B$ satisfies the inequalities (see \cite[Appendix V, Proposition 6]{Ber})
	\begin{equation*}
		\begin{split}
			&B(x)\leq\exp(2x\sqrt{\nu}/(\sqrt{\nu}-1)),\ 0\leq x\leq 1,\\
			&B(x)\leq B(1)x^{2\nu/(\nu-1)},\ x\geq 1.\\
		\end{split}
	\end{equation*}
	In particular, $\lim_{x\rightarrow0^{+}}B(x)=1$.
	
	Let $p=\frac{2n}{n-2}$ $,q=2$ and apply the \cite[Appendix V, Theorem 3]{Ber}, we then have
	\begin{theorem}\label{T2}(\cite[Theorem 5.2]{Che})
		Let $(X,g)$ be a compact $n$-dimensional smooth Riemannian manifold such that for some constant $b>0$,
		$$Ric(g)D^{2}(g)\geq-(n-1)b^{2}.$$
		If $f\in L^{2}_{1}(X)$ is a nonnegative continuous functions such that $f\De_{d} f\leq cf^{2}$, (here $\De_{d}=d^{\ast}d$), in the sense of distribution for  some positive constant number $c$, we then have
		\begin{equation*}
			\max_{x\in X}|f|^{2}(x)\leq B(\sigma_{n}R(b)c^{\frac{1}{2}} )\frac{\int_{X}f^{2} }{{\rm{Vol}}(g)}.
		\end{equation*} 
	\end{theorem}
We can obtain the $C^{0}$-estimate for harmonic $1$-forms on compact manifolds with a lower bound on the Ricci curvature.
	\begin{proposition}(\cite[Lemma 5.3]{Che})\label{L4}
		Let $(X,g)$ be a compact $n$-dimensional smooth Riemannian manifold, let $\theta$ be a harmonic $1$-form on $X$. If Ricci curvature of $X$ satisfies $Ric(g)\geq-a^{2}$, we then have
		\begin{equation}\label{E1}
			\max_{x\in X}|\theta^{\sharp}|^{2}(x)\leq B(\frac{\sigma_{n}Da}{bC(b)})
			\frac{\int_{X}|\theta^{\sharp}|^{2}}{{\rm{Vol}}(g)}.
		\end{equation}
		where $b=\frac{aD}{\sqrt{n-1}}$,  $\theta^{\sharp}$ is the dual vector field of $\theta$.  Furthermore, there exist a positive constant $c_{2}(n)$ such that if $aD\leq c_{2}(n)$, then 
		\begin{equation}\label{E2}
			\max_{x\in X}|\theta^{\sharp}|^{2}(x)\leq 2
			\frac{\int_{X}|\theta^{\sharp}|^{2}}{{\rm{Vol}}(g)}.
		\end{equation}
	\end{proposition}
	\begin{proof}
	Applying  the Bochner formula to $|\theta^{\sharp}|^{2}$, we get (see \cite[Section 7.3, Proposition 3.3]{Pet})
		\begin{equation*}
			\begin{split}
				-\frac{1}{2}\De_{d}|\theta^{\sharp}|^{2}&=|\na \theta^{\sharp}|^{2}+Ric(\theta^{\sharp},\theta^{\sharp})\\
				&\geq|\na \theta^{\sharp}|^{2}-a^{2}|\theta^{\sharp}|^{2}.\\
			\end{split}
		\end{equation*} 
		Following the second Kato inequality (cf. \cite{Ber}), we get
		$$|\theta^{\sharp}|\De_{d}|\theta^{\sharp}|\leq a^{2}|\theta^{\sharp}|^{2}.$$
		Apply Theorem \ref{T2} to $|\theta^{\sharp}|^{2}$, we have
		$$\max_{x\in X}|\theta^{\sharp}|^{2}(x)\leq B(\frac{\sigma_{n}Da}{bC(b)})
		\frac{\int_{X}|\theta^{\sharp}|^{2}}{{\rm{Vol}}(g)}.$$
		Following Lemma \ref{L2}, we observe that as $aD\rightarrow 0^{+}$,  $\frac{\sigma_{n}Da}{bC(b)}\leq\sigma_{n}a^{-1}_{n}(aD)\rightarrow 0$, it implies that
		$$B(\frac{\sigma_{n}Da}{bC(b)})\rightarrow 1.$$
		Therefore, we can choose a constant $c_{2}(n)$ such that if $aD\leq c_{2}(n)$, we have
		$$B(\frac{\sigma_{n}Da}{bC(b)})\leq 2.$$
	\end{proof}
We can easily observe that $\mathcal{H}^{k}_{\theta}(X)\subset\ker\mathcal{D}_{\theta}$, and thus we can also estimate harmonic forms with respect to $\De_{\theta}$ accordingly.
\begin{corollary}\label{C6}
	Let $(X,g)$ be a compact $n$-dimensional smooth Riemannian manifold, $\theta$ be a harmonic $1$-form on $X$ in $X\in\mathfrak{X}(p,k,D,\La)$. For each $\De_{\theta}$-harmonic $k$-form $\a$, $k\leq p$, we then have
	\begin{equation}
	\|\a\|_{L^{2\nu^{N}}(X)}\leq(\frac{\int_{X}|\a|^{2}}{{\rm{Vol}}(g)})^{\frac{1}{2}}\exp(NC_{3}(n)R(\La)\sqrt{\la})({\rm{Vol}}(g))^{\frac{1}{2\nu^{N} }},\ \forall N\in\mathbb{N}^{+},\\
	\end{equation}
	where  $\la:=2C_{2}^{2}(n)\max|\theta^{\sharp}|^{2}-c(n)\kappa$, $\nu=\frac{n}{n-2}$. Furthermore, we have
	\begin{equation}\label{E5}
	\int_{X}|f-\bar{f}||\a|^{2}\leq 2C_{4}(n)R(\La)(\frac{\int_{X}|\a|^{2}}{{\rm{Vol}}(g)})\exp(2NC_{3}(n)R(\La)\sqrt{\la})({\rm{Vol}}(g))^{\frac{1}{2 }} \|\theta^{\sharp}\|_{L^{\infty}(X)}\|\na \theta^{\sharp}\|_{L^{2}(X)},\\	
	\end{equation}
	and
	\begin{equation}\label{E6}
	\int_{X}|\theta^{\sharp}|^{2}|\a|^{2}\leq C_{1}(n)(\int_{X}|\na\theta^{\sharp}|^{2})^{\frac{1}{2}}(\frac{\int_{X}|\a|^{2}}{{\rm{Vol}}(g)})\exp(2NC_{3}(n)R(\La)\sqrt{\la})({\rm{Vol}}(g))^{\frac{1}{2}},
	\end{equation}	
	where $f=|\theta^{\sharp}|^{2}$, $\bar{f}=\frac{\int_{X}|\theta^{\sharp}|^{2} }{{\rm{Vol}}(g)}$, $N$ is an integer such that $2\nu^{N-1}<4\leq2\nu^{N}$, $\nu=\frac{n}{n-2}$.
\end{corollary}
\begin{proof}
The proof follows from Theorem \ref{T8} and Lemma \ref{L7}.
\end{proof}
Using the integral formula of $\De_{\theta}$ harmonic forms which established in Theorem \ref{T3}, we then have
\begin{lemma}(cf. \cite[Lemma 5.5]{Che})\label{L6}
Let $(X,g)$ be a compact $n$-dimensional smooth Riemannian manifold in  $\mathfrak{X}(p,\kappa,D,\La)$. If Ricci curvature of $X$ satisfies $Ric(g)\geq -a^{2}\geq -(n-1)\kappa$, then for each $\De_{t\theta}$-harmonic $k$-form $\a$, $k\leq p$, we have
		\begin{equation}\label{E14}
			\int_{X}|f-\bar{f}||\a|^{2}\leq 2C_{4}(n)\frac{aD}{b C(b)}\sqrt{B(\frac{\sigma_{n}Da}{bC(b)})} \exp(2NC_{3}(n)R(b)\sqrt{\la})
			\frac{\int_{X}|\theta^{\sharp}|^{2}}{{\rm{Vol}}(g)}\int_{X}|\a|^{2},
		\end{equation}
		where $b=\frac{aD}{\sqrt{n-1}}$, $f=|\theta^{\sharp}|^{2}$, $\bar{f}=\frac{\int_{X}|\theta^{\sharp}|^{2} }{{\rm{Vol}}(g)}$, $\la:=2C_{2}^{2}(n)\max t^{2}|\theta^{\sharp}|^{2}-c(n)\kappa$, and
		\begin{equation}\label{E15}
			\int_{X}t^{2}|\theta^{\sharp}|^{2}|\a|^{2}\leq C_{1}(n)a\exp(2NC_{3}(n)R(b)\sqrt{\la})(\frac{\int_{X}t^{2}|\theta^{\sharp}|^{2}}{{\rm{Vol}}(g)})^{\frac{1}{2}}
			(\int_{X}|\a|^{2}).
		\end{equation}	
	\end{lemma}
	\begin{proof}
	Noting that $Ric(g)D^{2}(g)\geq-a^{2}D^{2}:=-(n-1)b^{2}$, i.e., $b=\frac{aD}{\sqrt{n-1}}$. Then by Theorem \ref{T1}, we can replace the variate $\La$ of the function $R(\cdot)$ in inequalities (\ref{E5})  and (\ref{E6}) to $b$. We observe  that 
		$$-\frac{1}{2}\De_{d}|\theta^{\sharp}|^{2}\geq |\na \theta^{\sharp}|^{2}-a^{2}|\theta^{\sharp}|^{2}.$$ 
		Then
		\begin{equation}\label{E28}
			\int_{X}|\na \theta^{\sharp}|^{2}\leq a^{2}\int_{X}|\theta^{\sharp}|^{2}.
		\end{equation}
		Combining (\ref{E1}), (\ref{E5}), (\ref{E6}) and (\ref{E28}), we can get 
		\begin{equation*}
		\begin{split}
		&\quad\int_{X}|f-\bar{f}||\a|^{2}\\
		&\leq 2C_{4}(n)R(b)(\frac{\int_{X}|\a|^{2}}{{\rm{Vol}}(g)})\exp(2NC_{3}(n)R(b)\sqrt{\la})({\rm{Vol}}(g))^{\frac{1}{2 }} \|\theta^{\sharp}\|_{L^{\infty}(X)}\|\na \theta^{\sharp}\|_{L^{2}(X)}\ (by\ (\ref{E5}))\\
		&\leq 	2C_{4}(n)R(b)(\frac{\int_{X}|\a|^{2}}{{\rm{Vol}}(g)})\exp(2NC_{3}(n)R(b)\sqrt{\la})({\rm{Vol}}(g))^{\frac{1}{2 }} \sqrt{B(\frac{\sigma_{n}Da}{bC(b)})}	(\frac{\int_{X}|\theta^{\sharp}|^{2}}{{\rm{Vol}}(g)})^{\frac{1}{2}} a\|\theta^{\sharp}\|_{L^{2}(X)}(by\ (\ref{E1}),(\ref{E28}))\\
		&\leq 2C_{4}(n)\frac{aD}{b C(b)}\sqrt{B(\frac{\sigma_{n}Da}{bC(b)})} \exp(2NC_{3}(n)R(b)\sqrt{\la})
		\frac{\int_{X}|\theta^{\sharp}|^{2}}{{\rm{Vol}}(g)}\int_{X}|\a|^{2},
		\end{split}
		\end{equation*}
		and
		\begin{equation*}
		\begin{split}
		\int_{X}|\theta^{\sharp}|^{2}|\a|^{2}
		&\leq C_{1}(n)(\int_{X}|\na\theta^{\sharp}|^{2})^{\frac{1}{2}}(\frac{\int_{X}|\a|^{2}}{{\rm{Vol}}(g)})\exp(2NC_{3}(n)R(b)\sqrt{\la})({\rm{Vol}}(g))^{\frac{1}{2}}\ (by\ (\ref{E6}))\\
		&\leq  C_{1}(n)a(\int_{X}|\theta^{\sharp}|^{2})^{\frac{1}{2}}(\frac{\int_{X}|\a|^{2}}{{\rm{Vol}}(g)})\exp(2NC_{3}(n)R(b)\sqrt{\la})({\rm{Vol}}(g))^{\frac{1}{2}}(by\ (\ref{E28}))\\
		&\leq C_{1}(n)a\exp(2NC_{3}(n)R(b)\sqrt{\la})(\frac{\int_{X}|\theta^{\sharp}|^{2}}{{\rm{Vol}}(g)})^{\frac{1}{2}}
		(\int_{X}|\a|^{2}).\\
		\end{split}
		\end{equation*}	
	\end{proof}
	\begin{proof}[\textbf{Proof of Theorem \ref{T4}}]
		We can let $t\in\mathbb{R}$ such that  $$\frac{\int_{X}t^{2}|\theta^{\sharp}|^{2}}{{\rm{Vol}}(g)}=-16(n-1)C^{2}_{1}(n)\kappa.$$ Provided that 
		\begin{equation}\label{E13}
			aD\leq\min\{\sqrt{n-1}c_{1}(n), c_{2}(n)\},
		\end{equation}
		where $c_{2}(n)$ (resp. $c_{1}(n)$) is positive constant in Proposition \ref{L4} (resp. Lemma \ref{L2}), we then have
		$$t^{2}\max|\theta^{\sharp}|^{2}\leq -32(n-1)C_{1}^{2}(n)\kappa, \ (by\ (\ref{E2}))$$
		and 
		$$bC(b)\geq a_{n}\Rightarrow R(b)\leq\frac{D}{a_{n}}.$$
		Therefore, we get
	$$\la=2C_{2}^{2}(n)\max t^{2}|\theta^{\sharp}|^{2}-c(n)\kappa\leq -\kappa(64(n-1)C^{2}_{1}(n)C_{2}^{2}(n)+c(n))$$
		and
		\begin{equation}\label{E27}
		\exp(2NC_{3}(n)R(b)\sqrt{\la})\leq \exp(2a^{-1}_{n}C_{3}(n)N\La\sqrt{64(n-1)C_{1}^{2}(n)C_{2}^{2}(n)+c(n)}):=B(n,\La)
		\end{equation}
		By Lemma \ref{L6}, we have
		\begin{equation*}
			\begin{split}
				\int_{X}t^{2}|f-\bar{f}||\a|^{2}&\leq 2C_{4}(n)\frac{aD}{b C(b)} \exp(2NC_{3}(n)R(b)\sqrt{\la})\frac{\int_{X}t^{2}|\theta^{\sharp}|^{2}}{{\rm{Vol}}(g)}\int_{X}|\a|^{2}\\
				&\leq 2a^{-1}_{n}C_{4}(n)aDB(n,\La)\frac{\int_{X}t^{2}|\theta^{\sharp}|^{2}}{{\rm{Vol}}(g)}\int_{X}|\a|^{2},
			\end{split}
		\end{equation*}
		and
		\begin{equation*}
			\begin{split}
				\int_{X}t^{2}|\theta^{\sharp}|^{2}|\a|^{2}&\leq C_{1}(n)a\exp(2NC_{3}(n)R(b)\sqrt{\la})(\frac{\int_{X}t^{2}|\theta^{\sharp}|^{2}}{{\rm{Vol}}(g)})^{\frac{1}{2}}
				(\int_{X}|\a|^{2})\\
				&\leq C_{1}(n)a B(n,\La)(-16(n-1)C_{1}^{2}(n)\kappa)^{\frac{1}{2}}(\int_{X}|\a|^{2}).
			\end{split}
		\end{equation*}
		Provided that 
		$$2a^{-1}_{n}C_{4}(n)aDB(n,\La)\leq\frac{1}{4},\ i.e.,\ aD\leq\frac{a_{n}}{8C_{4}(n)B(n,\La)}$$
		and $$C_{1}(n)a B(n,\La)(-16(n-1)C_{1}^{2}(n)\kappa)^{\frac{1}{2}}\leq\frac{1}{2}(-16(n-1)C_{1}^{2}(n)\kappa),\ i.e.,\ aD\leq\frac{2\sqrt{n-1}\La}{B(n,\La)}.$$
		Therefore, we provided that
		\begin{equation}\label{E16}
		 aD\leq \min\{\sqrt{n-1}c_{1}(n), c_{2}(n) ,\frac{a_{n}}{8C_{4}(n)B(n,\La)}, \frac{2\sqrt{n-1}\La}{B(n,\La)}\},
		\end{equation} 
          then 
          $$
         \int_{X}t^{2}|f-\bar{f}||\a|^{2}\leq\frac{1}{4} \frac{\int_{X}t^{2}|\theta^{\sharp}|^{2}}{{\rm{Vol}}(g)}\|\a\|^{2}_{L^{2}(X)},$$
         and
         $$\int_{X}t^{2}|\theta^{\sharp}|^{2}|\a|^{2} 
          \leq\frac{1}{2} \frac{\int_{X}t^{2}|\theta^{\sharp}|^{2}}{{\rm{Vol}}(g)}\|\a\|^{2}_{L^{2}(X)}.$$
         We then obtain
		\begin{equation*}
			\begin{split}
				\frac{\int_{X}t^{2}|\theta^{\sharp}|^{2}}{{\rm{Vol}}(g)}\|\a\|^{2}_{L^{2}(X)}&\leq \int_{X}t^{2}|f-\bar{f}||\a|^{2}+\int_{X}t^{2}|\theta^{\sharp}|^{2}|\a|^{2}\\
				&\leq\frac{3}{4}\frac{\int_{X}t^{2}|\theta^{\sharp}|^{2}}{{\rm{Vol}}(g)}\|\a\|^{2}_{L^{2}(X)}.\\
			\end{split}
		\end{equation*}
		It implies that $\a=0$.
	\end{proof}
\begin{remark}	
If a sequence of Riemannian metrics $\{g_{i}\}_{n\in\mathbb{N}}$ on a smooth manifold $X$ satisfies
	$$Ric(g_{i})\geq-\frac{(n-1)}{i},\ and\  D(g_{i})\leq1$$
	for all $i\in\mathbb{N}$, then we say that the Riemannian manifold has almost nonnegative Ricci curvature. Here $Ric(g_{i})$ and $D(g_{i})$ stand for the Ricci curvature and diameter of $g_{i}$, respectively.
	
Suppose further that the curvature operators of the metrics $\{g_{i}\}$ are uniformly bounded from below i.e., there exists a negative constant $\kappa$ such that $$\la_{1}(X,g_{i})\geq\kappa, \forall i\in\mathbb{N},$$
where $\la_{1}(X,g_{i})$ is the smallest eigenvalue of the curvature operator. For sufficiently large $i$, the manifold $(X,g_{i})\in \mathfrak{X}(\lfloor\frac{n}{2}\rfloor,\kappa,1,\kappa)$, and the Ricci curvature of $g_{i}$ satisfies $$Ric(g_{i})D(g_{i})\geq-\frac{(n-1)}{i}\geq-C(n,\kappa).$$ 
where $C(n,\kappa)$ is a positive constant independent of $i$ that appears in Theorem \ref{T4} (see (\ref{E16})). Therefore, our conclusion generalizes Theorem 1.4 in \cite{Che}. 
	\end{remark}
	\begin{proof}[\textbf{Proof of Corollary \ref{T5}}]
	We observe that	as $\La\rightarrow0^{+}$, $B(n,\La)\rightarrow1$ (see (\ref{E27})). 
		Therefore, there exists a constant $c_{3}(n)$ such that if $\La\leq c_{3}(n)$, then
	$$B(n,\La)\leq2\Rightarrow\frac{2\sqrt{n-1}\La}{B(n,\La)}\geq\sqrt{n-1}\La,\ and\ \frac{a_{n}}{8C_{4}(n)B(n,\La)}\geq\frac{a_{n}}{16C_{4}(n)}.$$
	One can see that $aD\leq\sqrt{n-1}\La$. Therefore,  $aD\leq \frac{2\sqrt{n-1}\La}{B(n,\La)}$.  Provided  that	\begin{equation}\label{E31}
	aD\leq \min\{\sqrt{n-1}c_{1}(n), c_{2}(n) ,\sqrt{n-1}c_{3}(n)\frac{a_{n}}{16C_{4}(n)}\},\ (by \ (\ref{E16})),
	\end{equation} 
	then  there exists some $t\in\mathbb{R}$, $t\neq0$ such that $H^{k}(X,t\theta)=0$ for any $k\leq p$.
	\end{proof}
	A compact Riemannian manifold $(X,g)$ is said to be $\varepsilon$-flat if the Riemannian sectional curvature $|sec(g)|\leq K$ and the diameter $D(g)$ of $X$satisfy the inequality $KD^{2}\leq\varepsilon$.  Gromov proved that any sufficiently flat Riemannian manifold possesses a finite cover which is diffeomorphic to a nil-manifold. Furthermore, Gromov demonstrated that every nil-manifold carries an $\varepsilon$-flat metric for any $\varepsilon>0$ (see \cite{BK81,Gro,Ruh}).  A compact manifold $X$ is called almost flat if there exists a sequence of metrics $g_{i}$ on $X$ such that  $$K(g_{i})D^{2}(g_{i})\leq\frac{1}{i}.$$  
	By the Chern-Weil theory, it can be shown that the Pontryagin numbers of an oriented almost flat manifold $X$ all vanish. 
	\begin{proposition}(\cite[Proposition 3.8]{BK})\label{P2}
	For a  smooth Riemannian manifold $(X,g)$ with the sectional curvature $sec(g)$ of the metric $g$ satisfies
		$$a\leq sec(g)\leq b,$$ 
		then the eigenvalues of  its curvature operator lie in the interval 
		\begin{equation}	(a+b)-\frac{4\lfloor\frac{n}{2}\rfloor-1}{3}(b-a)\leq\la\leq(a+b)+\frac{4\lfloor\frac{n}{2}\rfloor-1}{3}(b-a).
		\end{equation}
		Moreover the last upper bound is achieved on $\C P^{n}$.
	\end{proposition}
	\begin{corollary}\label{C3}
	Let $X$ be a compact $n$-dimensional smooth Riemannian $\varepsilon$-flat manifold with nonzero first de Rham cohomology group.  Then there exists a uniform positive constant $C(n)$ such that if
		$$\varepsilon\leq C(n),$$
then there exists some $t\in\mathbb{R}$, $t\neq0$ such that $H^{p}(X,t\theta)=0$ for any $p$, where $H^{p}(X,t\theta)$ is the Morse-Novikov cohomology group with respect to $t\theta$. In particular, the Euler number of $X$ satisfies $\chi(X)=0$.
	\end{corollary}
	\begin{proof}
		Following Proposition \ref{P2}, the smallest eigenvalue of curvature operator satisfies
		$$\la_{1}\geq\frac{-8\lfloor \frac{n}{2}\rfloor+2 }{3}K.$$
	For $\varepsilon\leq\frac{3}{8\lfloor \frac{n}{2}\rfloor-2 } C(n)$, we have $$-\la_{1}D^{2}\leq C(n),$$ 
	where $C(n)$ is the constant in Corollary \ref{T5}. Then, according to Corollary  $\ref{T5}$, it implies that there exists some $t\in\mathbb{R}$, $t\neq0$ such that the Morse-Novikov $p$-th cohomology group $H^{p}(X,t\theta)$ with respect to $t\theta$ vanishes for any $p$.
	\end{proof}
	\begin{remark}
A Riemannian manifold $X$ is said to have almost nonnegative sectional curvature if it admits a sequence of Riemannian metrics $g_{i}$ such that
		$$sec(g_{i})\geq-\frac{1}{i},\ and,\ D(g_{i})\leq 1.$$
An almost flat manifold is a special case of almost nonnegative sectional curvature.  Using collapsing theory, Yamaguchi \cite{Yam} proved that any compact almost nonnegative sectional curvature with nonzero first de Rham cohomology group has a vanishing Euler number. Chen established the vanishing theorem for Morse-Novikov cohomology group, reaffirming this result (see \cite[Theorem 1.1]{Che}). However, our method of proving Corollary \ref{T5} is not applicable in the case of an almost nonnegative manifold, as we have no way to show that the sectional curvature of $g_{i}$ is uniformly bounded from above.
	\end{remark}

	\subsection{Killing vector field and Euler number}
A vector field $\theta^{\sharp}$ on a Riemannian manifold $(X,g)$ is called a Killing field if $L_{\theta^{\sharp}}g=0$, meaning that $\theta^{\sharp}$ leaves the metric $g$ invariant. This implies that the local flows generated by $\theta^{\sharp}$ act as isometries. The dual $1$-form of $\theta^{\sharp}$ is denoted by $\theta$. 

We can obtain the $C^{0}$-estimate for the Killing vector fields on compact manifolds in $\mathfrak{X}(p,\kappa,D,\La)$ with an upper bound on the Ricci curvature. 
	\begin{proposition}(cf. \cite{CH})
	Let  $(X,g)$ be a compact smooth Riemannian manifold , $\theta^{\sharp}$ be a Killing vector field  on $X$. If the Ricci curvature satisfies 
		$$ Ric(g)\leq a^{2} ,$$	
		then 
		\begin{equation}\label{E21}
			\|\na\theta^{\sharp}\|^{2}_{L^{2}(X)}\leq a^{2}\|\theta^{\sharp}\|^{2}_{L^{2}(X)}.
		\end{equation}
		Furthermore, if  $X\in\mathfrak{X}(p,\kappa,D,\La)$, then
		\begin{equation}\label{E20}
			\max|\theta^{\sharp}|^{2}\leq B(\sigma_{n}\frac{aD}{\La C(\La)})\frac{\int_{X}|\theta^{\sharp}|^{2} }{ {\rm{Vol}}(g)}.
		\end{equation}
	\end{proposition}
	\begin{proof}
		Applying  Bochner formula to $|\theta^{\sharp}|^{2}$, we get (see \cite[Proposition 1.4]{Pet})
		\begin{equation*}
			\begin{split}
				-\frac{1}{2}\De_{d}|\theta^{\sharp}|^{2}&=|\na \theta^{\sharp}|^{2}-Ric(\theta^{\sharp},\theta^{\sharp})\\
				&\geq|\na \theta^{\sharp}|^{2}-a^{2}|\theta^{\sharp}|^{2}.
			\end{split}
		\end{equation*}
		Following the second Kato inequality, we get $|\theta^{\sharp}|\De_{d}|\theta^{\sharp}|\leq a^{2}|\theta^{\sharp}|^{2}$.  Apply Theorem \ref{T2} to $|\theta^{\sharp}|^{2}$, we obtain the estimate (\ref{E20}).
	\end{proof}

	\begin{lemma}\label{L3}
		Let $(X,g)$ be a compact $n$-dimensional smooth Riemannian manifold in  $\mathfrak{X}(\lfloor\frac{n}{2}\rfloor,\kappa,D,\La)$, $\theta^{\sharp}$ a Killing vector field on $X$. If Ricci curvature $Ric(g)\leq a^{2}$, then for each $\a\in\ker\mathcal{D}_{t\theta}\cap\Om^{\pm}(X)$, we have
		\begin{equation}\label{E18}
			\int_{X}|f-\bar{f}||\a|^{2}\leq 2C_{4}(n)\frac{aD}{\La C(\La)}\sqrt{B(\frac{\sigma_{n}Da}{\La C(\La)})} \exp(2NC_{3}(n)R(\La)\sqrt{\la})
			\frac{\int_{X}|\theta^{\sharp}|^{2}}{{\rm{Vol}}(g)}\int_{X}|\a|^{2},
		\end{equation}
		where $f=|\theta^{\sharp}|^{2}$, $\bar{f}=\frac{\int_{X}|\theta^{\sharp}|^{2} }{{\rm{Vol}}(g)}$, $\la=2C_{2}^{2}(n)\max t^{2}|\theta^{\sharp}|^{2}-c(n)\kappa$, and
		\begin{equation}\label{E19}
			\int_{X}t^{2}|\theta^{\sharp}|^{2}|\a|^{2}\leq C_{1}(n)a\exp(2NC_{3}(n)R(\La)\sqrt{\la})(\frac{\int_{X}t^{2}|\theta^{\sharp}|^{2}}{{\rm{Vol}}(g)})^{\frac{1}{2}}
			(\int_{X}|\a|^{2}).
		\end{equation}	
	\end{lemma}
	\begin{proof}
		The proof is similar to Lemma \ref{L6}.
	\end{proof}
	\begin{proof}[\textbf{Proof of Theorem \ref{T6}}]
		We let $t\in\mathbb{R}$ such that $$\frac{\int_{X}t^{2}|\theta^{\sharp}|^{2}}{{\rm{Vol}}(g)}=-\kappa.$$ 
	Provided that 
	$$B(\frac{\sigma_{n}Da}{\La C(\La)})\leq2,\ i.e., aD\leq \ln2\sigma^{-1}_{n}\La C(\La)(\sqrt{\nu}-2)\sqrt{\nu}^{-1}, $$
	we then have
	$$t^{2}\max|\theta^{\sharp}|^{2}\leq -2\kappa,\ (by\ (\ref{E20}))$$
	and
	$$\la=2C_{2}^{2}(n)\max t^{2}|\theta^{\sharp}|^{2}-c(n)\kappa\leq -\kappa(4C_{2}^{2}(n)+c(n)).$$
	Therefore, we get
	$$\exp(2NC_{3}(n)R(\La)\sqrt{\la})\leq \exp(2C_{3}(n)N\sqrt{4C_{2}^{2}(n)+c(n)}C^{-1}(\La)):=\tilde{B}(n,\La)$$
	Provided that $$aD\leq\frac{\La C(\La)}{6\sqrt{2}C(n)\tilde{B}(n,\La)}
	\Rightarrow2C(n)\frac{aD}{\La C(\La)}\sqrt{B(\frac{\sigma_{n}Da}{\La C(\La)})}\tilde{B}(n,\La)\leq\frac{1}{3},$$
	and $$C_{n}a\tilde{B}(n,\La)(-\kappa)^{\frac{1}{2}}\leq-\frac{1}{3}\kappa,\ i.e.,\ aD\leq\frac{\La}{3C_{n}\tilde{B}(n,\La)}.$$
	Therefore, we provided that
	$$aD\leq\{\ln2\sigma^{-1}_{n}\La C(\La)(\sqrt{\nu}-2)\sqrt{\nu}^{-1}, \frac{\La C(\La)}{6\sqrt{2}C(n)\tilde{B}(n,\La)},\frac{\La}{3C_{n}\tilde{B}(n,\La)} \},
	$$
	then
	\begin{equation*}
	\begin{split}
	\frac{\int_{X}t^{2}|\theta^{\sharp}|^{2}}{{\rm{Vol}}(g)}\|\a\|^{2}_{L^{2}(X)}&\leq \int_{X}t^{2}|f-\bar{f}||\a|^{2}+\int_{X}t^{2}|\theta^{\sharp}|^{2}|\a|^{2}\\
	&\leq\frac{2}{3}\frac{\int_{X}t^{2}|\theta^{\sharp}|^{2}}{{\rm{Vol}}(g)}\|\a\|^{2}_{L^{2}(X)}.\\
	\end{split}
	\end{equation*}
	It implies that $\a=0$.
	
	\end{proof}
	
	\begin{proof}[\textbf{Proof of Corollary \ref{T7}}]
	For both cases, the manifold $X\in\mathfrak{X}(\lfloor\frac{n}{2}\rfloor,\kappa,D,\La)$, where $\kappa$ can be any negative constant. Noting that $\liminf_{\La\rightarrow0^{+}}\La C(\La)\geq a_{n}$. We now let $\kappa\rightarrow0^{-}$ i.e., $\La\rightarrow 0^{+}$, in Lemma (\ref{L3}), then for any $\a\in\ker\mathcal{D}_{t\theta}$, we have
			\begin{equation}\label{E29}
		\int_{X}|f-\bar{f}||\a|^{2}\leq 2C_{4}(n)\frac{aD}{a_{n}}\sqrt{B(\frac{\sigma_{n}Da}{a_{n}})} \exp(2NC_{3}(n)\frac{D\sqrt{\la}}{a_{n}})
		\frac{\int_{X}|\theta^{\sharp}|^{2}}{{\rm{Vol}}(g)}\int_{X}|\a|^{2},
		\end{equation}
		where $f=|\theta^{\sharp}|^{2}$, $\bar{f}=\frac{\int_{X}|\theta^{\sharp}|^{2} }{{\rm{Vol}}(g)}$, and
		\begin{equation}\label{E30}
		\int_{X}t^{2}|\theta^{\sharp}|^{2}|\a|^{2}\leq C_{1}(n)a\exp(2NC_{3}(n)\frac{D\sqrt{\la}}{a_{n}})(\frac{\int_{X}t^{2}|\theta^{\sharp}|^{2}}{{\rm{Vol}}(g)})^{\frac{1}{2}}
		(\int_{X}|\a|^{2}).
		\end{equation}	
		where $\la:=2C_{2}^{2}(n)\max t^{2}|\theta^{\sharp}|^{2}$. We can choose $t\neq 0$ such that 
		$$\frac{\int_{X}t^{2}|\theta^{\sharp}|^{2}}{{\rm{Vol}}(g)}=36C_{1}^{2}(n)a^{2},$$
		There exist a positive constant $c_{4}(n)$ such that if $aD\leq c_{4}(n)$, then
		$$\sqrt{B(\frac{\sigma_{n}Da}{a_{n}})}\leq2$$
		and
		$$\exp(2NC_{3}(n)Da_{n}^{-1}\sqrt{\la})\leq2,$$
		then by  (\ref{E30}), we get 
		\begin{equation*}
			\begin{split}
				\int_{X}t^{2}|\theta^{\sharp}|^{2}|\a|^{2}
				&\leq C_{1}(n)a\exp(2NC_{3}(n)Da_{n}^{-1}\sqrt{\la})(\frac{\int_{X}t^{2}|\theta^{\sharp}|^{2}}{{\rm{Vol}}(g)})^{\frac{1}{2}}(\int_{X}|\a|^{2})\\
				&\leq 2C_{1}(n) a\|\a\|^{2}_{L^{2}(X)}
				\sqrt{ \frac{\int_{X}t^{2}|\theta^{\sharp}|^{2}}{{\rm{Vol}}(g)}}
				(\int_{X}|\a|^{2})\\
				&\leq 12C^{2}_{1}(n)a^{2}\|\a\|_{L^{2}(X)}\\
				&=\frac{1}{3}\frac{\int_{X}t^{2}|\theta^{\sharp}|^{2}}{{\rm{Vol}}(g)}\|\a\|^{2}_{L^{2}(X)}.\\
			\end{split}
		\end{equation*}
		Provided that 
		$aD\leq\min\{c_{4}(n),\frac{a_{n}}{24C_{4}(n)}\}$, then by (\ref{E18}),
		\begin{equation*}
			\begin{split}
				\int_{X}t^{2}|f-\bar{f}||\a|^{2}&\leq 2C_{4}(n)\frac{aD}{a_{n}}\sqrt{B(\frac{\sigma_{n}Da}{a_{n}})} \exp(2NC_{3}(n)\frac{D\sqrt{\la}}{a_{n}})
				\frac{\int_{X}t^{2}|\theta^{\sharp}|^{2}}{{\rm{Vol}}(g)}\int_{X}|\a|^{2}\\
				&\leq 8C_{4}(n)a^{-1}_{n}aD\frac{\int_{X}t^{2}|\theta^{\sharp}|^{2}}{{\rm{Vol}}(g)}\int_{X}|\a|^{2}\\
				&\leq\frac{1}{3}\frac{\int_{X}t^{2}|\theta^{\sharp}|^{2}}{{\rm{Vol}}(g)}\|\a\|^{2}_{L^{2}(X)}.\\
			\end{split}
		\end{equation*}
		Combining preceding inequalities yields
		\begin{equation*}
			\frac{\int_{X}t^{2}|\theta^{\sharp}|^{2}}{{\rm{Vol}}(g)}\|\a\|^{2}_{L^{2}(X)}\leq\frac{2}{3}\frac{\int_{X}t^{2}|\theta^{\sharp}|^{2}}{{\rm{Vol}}(g)}\|\a\|^{2}_{L^{2}(X)}.
		\end{equation*}
	Assume that the Killing filed is non-zero, then $\a=0$, it implies that $\chi(X)=0$. However, according to the \cite[Theorem A] {PW}and \cite[Page 36, Corollary]{PW},  the Euler number $\chi(X)=2$ under the conditions in Corollary \ref{T7}, hence we have $\theta^{\sharp}=0$.
	\end{proof}

  \section*{Acknowledgements}
	This work is supported by the National Natural Science Foundation of China Nos. 12271496 (Huang), 11701226 (Tan) and the Youth Innovation Promotion Association CAS, the Fundamental Research Funds of the Central Universities, the USTC Research Funds of the Double First-Class Initiative.
	
\noindent\textbf{Data availability} {This manuscript has no associated data.}
	\section*{Declarations}
\noindent\textbf{Conflict of interest} The author states that there is no conflict of interest.

\end{document}